\documentclass[a4paper]{amsart}

\usepackage{amsmath} % Lots of maths functionality
\usepackage{amssymb} % Maths symbols
\usepackage{amsthm} % Maths environments: \begin{proof}, etc.
\usepackage[english]{babel} % Language and hyphenation.
\usepackage[font=small,justification=centering]{caption} % More flexibility for captioning figures
\usepackage[nodayofweek]{datetime}
\usepackage[shortlabels]{enumitem} % Change enumeration labelling with \begin{enumerate}[a)] etc.
\usepackage[T1]{fontenc} % For font encoding, to allow accents, copy & paste, inequality signs, etc. to all work nicely.
\usepackage[a4paper]{geometry} % Changing page shape
\usepackage[utf8]{inputenc} % To be loaded after fontenc, also for encoding.
\usepackage{ifthen} % For Dani's \begin{com} environment and \numberedtheorem
\usepackage{mathabx} % Contains \pnest symbol and more. Clashes with accents for \ring
\usepackage{mathtools} % Uses amsmath, fixes quirks and adds functionality.
\usepackage[dvipsnames]{xcolor} % Allow links to have colour. Needs to be before hyperref and tikz-cd
\usepackage[pdftex,  colorlinks=true]{hyperref} % Makes references and citations into links
    \hypersetup{urlcolor=RoyalBlue, linkcolor=RoyalBlue,  citecolor=black}
\usepackage{setspace} % Allows \onehalfspacing etc. for changing gaps between lines
\usepackage{tikz-cd} % Commutative diagrams
\usepackage{xfrac} % Nicer quotients 
\usepackage{cleveref}
\makeatletter
\def\@tocline#1#2#3#4#5#6#7{\relax
  \ifnum #1>\c@tocdepth % then omit
  \else
    \par \addpenalty\@secpenalty\addvspace{#2}%
    \begingroup \hyphenpenalty\@M
    \@ifempty{#4}{%
      \@tempdima\csname r@tocindent\number#1\endcsname\relax
    }{%
      \@tempdima#4\relax
    }%
    \parindent\z@ \leftskip#3\relax \advance\leftskip\@tempdima\relax
    \rightskip\@pnumwidth plus4em \parfillskip-\@pnumwidth
    #5\leavevmode\hskip-\@tempdima
      \ifcase #1
       \or\or \hskip 1em \or \hskip 2em \else \hskip 3em \fi%
      #6\nobreak\relax
    \dotfill\hbox to\@pnumwidth{\@tocpagenum{#7}}\par% <---- \dotfill -> \hfill
    \nobreak
    \endgroup
  \fi}
\makeatother

\newtheorem{thm}{Theorem}[section]
\newtheorem{lemma}[thm]{Lemma}

\newtheorem{thmx}{Theorem}
\newtheorem{corx}[thmx]{Corollary}

\theoremstyle{definition}

\newcommand{\cald}{{\mathcal{D}}}
\newcommand{\cale}{{\mathcal{E}}}
\newcommand{\calf}{{\mathcal{F}}}

\newcommand{\calm}{{\mathcal{M}}}

\newcommand{\calo}{{\mathcal{O}}}

\DeclareMathOperator{\cd}{\mathrm{cd}}

\DeclareMathOperator{\TC}{\sf{TC}}

\usepackage{tikz}
\usetikzlibrary{arrows,quotes}
\tikzstyle{blackNode}=[fill=black, draw=black, shape=circle]

\newcounter{commentcounter}

\newcommand{\F}{{\mathcal{F}}}
\newcommand{\E}{{\mathcal{E}}}
\newcommand{\calH}{{\mathcal{H}}}
\newcommand{\TR}{{\mathcal{TR}}}
\newcommand{\calD}{{\mathcal{D}}}
\newcommand{\spann}[1]{{\ensuremath \langle{#1}\rangle}}
\newcommand{\EFG}{E_\mathcal{F}G}
\newcommand{\OFG}{\mathcal{O}_\mathcal{F}G}

\usepackage[style=alphabetic, citestyle=alphabetic, sorting=nyvt]{biblatex}
\AtEveryBibitem{\ifentrytype{article}{\clearfield{issn}}{}}
\usepackage{csquotes}

\NewBibliographyString{toappear}
\DefineBibliographyStrings{english}{%
  toappear = {to appear},
}

\renewbibmacro*{in:}{%
  \iffieldundef{pubstate}
    {}
    {\printfield{pubstate}%
     \setunit{\addspace}%
     \clearfield{pubstate}}%
  \printtext{%
    \bibstring{in}\intitlepunct}}
    
\title{Higher topological complexity of hyperbolic groups}
\author{Sam Hughes}
\address[S.\ Hughes]{Mathematical Institute, University of Oxford, Oxford OX2 6GG, UK}
\email{sam.hughes@math.ox.ac.uk}
\author{Kevin Li}
\address[K.\ Li]{School of Mathematical Sciences, University of Southampton, Southampton SO17 1BJ, UK}
\email{kevin.li@soton.ac.uk}

\date{11th October, 2021}

\begin{document}

\maketitle

\begin{abstract}
We prove for non-elementary torsion-free hyperbolic groups $\Gamma$ and all $r\ge 2$ that the higher topological complexity ${\sf{TC}}_r(\Gamma)$ is equal to $r\cdot \mathrm{cd}(\Gamma)$.  In particular, hyperbolic groups satisfy the rationality conjecture on the $\sf{TC}$-generating function, giving an affirmative answer to a question of Farber and Oprea.  More generally, we consider certain toral relatively hyperbolic groups.
\end{abstract}

\section{Introduction}

Let $r\ge 2$ be an integer. The higher (or sequential) topological complexity $\TC_r(X)$ of a path-connected space $X$ was introduced by Rudyak~\cite{Rudyak2010}, generalising Farber's topological complexity~\cite{Farber03}. The motivation for these numerical invariants arises from robotics. They provide a measure of complexity for the motion planning problem in the configuration space $X$ with prescribed initial and final states, as well as $r-2$ consecutive intermediate states.
More precisely, consider the path-fibration $p\colon X^{[0,1]}\to X^r$ that maps a path $\omega\colon [0,1]\to X$ to the tuple $(\omega(0),\omega(\frac{1}{r-1}),\ldots,\omega(\frac{r-2}{r-1}),\omega(1))$.
Then $\TC_r(X)$ is defined as the minimal integer $n$ for which $X^r$ can be covered by $n+1$ many open subsets $U_0,\ldots,U_n$ such that $p$ admits a local section over each $U_i$. If no such $n$ exists, one sets $\TC_r(X)\coloneqq \infty$. Note that $\TC_2(X)$ recovers the usual topological complexity.

Since the higher topological complexities are homotopy invariants, one obtains interesting invariants of groups $\Gamma$ by setting $\TC_r(\Gamma)\coloneqq \TC_r(K(\Gamma,1))$, where $K(\Gamma,1)$ is an Eilenberg--MacLane space. The topological complexities $\TC_r(\Gamma)$ have been computed for several classes of groups (see e.g.~\cite{Farber19,Farber-Mescher20,Dranishnikov20} for $r=2$, \cite{FO2019,AGGO20,GGY16} for $r\ge 2$, and references therein). In a celebrated result of Dranishnikov~\cite{Dranishnikov20} (see also~\cite{Farber-Mescher20}), the topological complexity $\TC_2(\Gamma)$ of groups with cyclic centralisers, such as hyperbolic groups, was shown to equal $\cd(\Gamma\times\Gamma)$. Here $\cd$ denotes the cohomological dimension. We generalise this result to all higher topological complexities $\TC_r(\Gamma)$ for $r\ge 2$, as well as to a larger class of groups containing certain toral relatively hyperbolic groups. 
Recall that a collection $\{P_i\ |\ i\in I\}$ of subgroups of $\Gamma$ is called \emph{malnormal}, if for all $g\in \Gamma$ and  $i,j\in I$ we have $gP_ig^{-1}\cap P_j=\{e\}$, unless $i=j$ and $g\in P_i$.

\begin{thmx}\label{thmx.main}
Let $r\geq 2$ and let $\Gamma$ be a torsion-free group with $\cd(\Gamma)\geq 2$.  Suppose that $\Gamma$ admits a malnormal collection of abelian subgroups $\{P_i\ |\ i\in I\}$ satisfying $\cd(P_i^r)<\cd(\Gamma^r)$ such that the centraliser $C_\Gamma(g)$ is cyclic for every $g\in\Gamma$ that is not conjugate into any of the $P_i$.  Then $\TC_r(\Gamma)=\cd(\Gamma^r)$.
\end{thmx}
\noindent The preceding theorem was obtained for the case $r=2$ by the second author in~\cite{Li2021}.

For a space $X$, the \emph{$\TC$-generating function} $f_X(t)$ is defined as the formal power series
\[
    f_X(t):=\sum_{r=1}^\infty\TC_{r+1}(X)\cdot t^r .
\]
The $\TC$-generating function of a group $\Gamma$ is set to be $f_\Gamma(t)\coloneqq f_{K(\Gamma,1)}(t)$.
Recall that a group $\Gamma$ is said to be \emph{of type $F$} (or \emph{geometrically finite}) if it admits a finite model for $K(\Gamma,1)$.

Following~\cite{FO2019}, we say that a finite CW-complex $X$ (resp.\ a group $\Gamma$ of type $F$) satisfies the \emph{rationality conjecture} if the $\TC$-generating function $f_X(t)$ (resp.\ $f_\Gamma(t)$) is a rational function of the form $\frac{P(t)}{(1-t)^2}$, where $P(t)$ is an integer polynomial with $P(1)=\operatorname{cat}(X)$ (resp.\ $P(1)=\cd(\Gamma)$). Here $\operatorname{cat}$ denotes the Lusternik--Schnirelmann category. While a counter-example to the rationality conjecture for finite CW-complexes was found in~\cite{FKS20}, the rationality conjecture for groups of type $F$ remains open. It is known to hold, e.g.\ for abelian groups of type $F$, right-angled Artin groups, fundamental groups of closed orientable surfaces, and Higman's group (see~\cite[Section 8]{FO2019}). 
Our result extends the class of groups for which the rationality conjecture holds as follows.

\begin{corx}\label{corx.growth}
Let $\Gamma$ be a group as in Theorem~\ref{thmx.main}.  If $\Gamma$ is of type $F$, then
\[f_\Gamma(t)=\cd(\Gamma)\frac{(2-t)t}{(1-t)^2}. \]
In particular, the rationality conjecture holds for $\Gamma$.
\end{corx}

As remarked by Farber and Oprea in~\cite[page~159]{FO2019}, it is particularly interesting to determine the validity of the rationality conjecture for the class of hyperbolic groups. We answer their question in the affirmative.

\begin{corx}\label{corx.hyp}
The rationality conjecture holds for torsion-free hyperbolic groups.
\end{corx}

\subsection*{Acknowledgements}
Both authors would like to thank the organisers of the workshop ``Advances in Homotopy Theory'' held in September 2021 at the Southampton Centre for Geometry, Topology, and Applications which inspired this work. The first author was supported by the Engineering and Physical Sciences Research Council grant number 2127970.

\section{Background}

We recall the notion of a classifying space for a family of subgroups (see e.g.~\cite{Lueck05survey}). A \emph{family $\F$ of subgroups} of a group $G$ is a non-empty set of subgroups that is closed under conjugation and finite intersections. The family consisting only of the trivial subgroup is denoted by $\TR$. The family $\F\spann{\calH}$ generated by a set of subgroups $\calH$ is the smallest family containing $\calH$. For a family $\F$ of subgroups of $G$ and a subgroup $H$ of $G$, we denote by $\F|_H$ the family $\{L\cap H\ |\ L\in \F\}$ of subgroups of $H$.
A \emph{classifying space $\EFG$ for the family $\F$} is a terminal object in the $G$-homotopy category of $G$-CW-complexes with isotropy groups in $\F$. Note that a model for $E_\TR G$ is given by $EG$ and in particular, there exists a $G$-map $EG\to \EFG$ that is unique up to $G$-homotopy.

Let $\E\subset \F$ be two families of subgroups of $G$. We say that $G$ satisfies condition $(M_{\E\subset \F})$ if every element in $\F\setminus \E$ is contained in a unique maximal element $M\in \F\setminus \E$, and that $G$ satisfies condition $(NM)_{\E\subset \F}$ if additionally $M$ equals its normaliser $N_G(M)$. We denote the Weyl group by $W_G(M)\coloneqq N_G(M)/M$.
The following proposition is a special case of a construction due to L\"uck and Weiermann~\cite[Corollary~2.8]{Lueck-Weiermann12} stated in \cite[Corollary~2.2]{Li2021}.

\begin{thm}[L\"uck--Weiermann]\label{prop.LW}
	Let $G$ be a group and $\cale\subset\calf$ be two families of subgroups. Let $\{M_i\ |\ i\in I\}$ be a complete set of representatives for the conjugacy classes of maximal elements in $\F\setminus \E$.
\begin{enumerate}[label={\rm{(\roman*)}}]
	\item\label{prop.LW.TRIV} If $\E=\TR$ and $G$ satisfies condition $(M_{\TR\subset \calf})$, then a model for $E_\calf G$ is given by the following $G$-pushout
	\[\begin{tikzcd}
		\coprod_{i\in I} G\times_{N_G(M_i)}E(N_G(M_i))\ar{r}\ar{d} & EG\ar{d} \\
		\coprod_{i\in I} G\times_{N_G(M_i)}E(W_G(M_i)) \ar{r} & E_\calf G ;
	\end{tikzcd}\]

	\item\label{prop.LW.NM} If $G$ satisfies conditions $(M_{\cale\subset \calf})$ and $(NM_{\cale\subset \calf})$, then a model for $E_\calf G$ is given by the following $G$-pushout
	\[\begin{tikzcd}
		\coprod_{i\in I} G\times_{M_i}E_{\cale|_{M_i}} M_i\ar{r}\ar{d} & E_\cale G\ar{d} \\
		\coprod_{i\in I} G/M_i \ar{r} & E_\calf G .
	\end{tikzcd}\]
\end{enumerate}
\end{thm}

The \emph{$\F$-restricted orbit category $\OFG$} has $G$-sets $G/H$ with $H\in \F$ as objects and $G$-maps as morphisms. Let $A$ be an \emph{$\OFG$-module}, that is a contravariant functor from the orbit category $\OFG$ to the category of modules.
The equivariant cellular cohomology $H^*_G(X;A)$ of a $G$-CW-complex $X$ with isotropy groups in $\F$ is called \emph{Bredon cohomology}~\cite{Bredon67}. In particular, Bredon cohomology satisfies the Mayer--Vietoris axiom for $G$-pushouts.

The (higher) topological complexity $\TC_r(\Gamma)$ of a group $\Gamma$ for $r\ge 2$ can be characterised in terms of classifying spaces for families~\cite[Theorem 3.1]{FO2019}, generalising a result of Farber--Grant--Lupton--Oprea~\cite[Theorem 3.3]{Farber19} for the case $r=2$. Consider $G=\Gamma^r$ and let $\calD$ be the family of subgroups that is generated by the diagonal subgroup $\Delta(\Gamma)\subset \Gamma^r$.
\begin{thm}[Farber--Oprea]\label{thm.Farber-Oprea}
    Let $\Gamma$ be a group and $r\ge 2$. Then $\TC_r(\Gamma)$ equals the infimum of integers $n$ for which the canonical $\Gamma^r$-map
    \[
        E(\Gamma^r)\to E_\calD(\Gamma^r)
    \]
    is $\Gamma^r$-equivariantly homotopic to a $\Gamma^r$-map with values in the $n$-skeleton $E_\calD(\Gamma^r)^{(n)}$.
\end{thm}

As a consequence~\cite[Theorem 5.1]{FO2019}, a lower bound for $\TC_r(\Gamma)$ is given by the supremum of integers $n$ for which the canonical $\Gamma^r$-map $E(\Gamma^r)\to E_\calD(\Gamma^r)$ induces a non-trivial map in Bredon cohomology 
\[
    H^n_{\Gamma^r}(E_\calD(\Gamma^r);A)\to H^n_{\Gamma^r}(E(\Gamma^r);A)
\] 
for some $\mathcal{O}_\calD(\Gamma^r)$-module $A$.

\section{Proofs}

Fix an integer $r\ge 2$.
Let $\Gamma$ be a group and $\Delta\colon \Gamma\to \Gamma^r$ be the diagonal map.
For $\gamma=(\gamma_1,\ldots,\gamma_{r-1})\in \Gamma^{r-1}$ and a subset $S\subset \Gamma$, we define the subgroup $H_{\gamma,S}$ of $\Gamma^r$ as
\[
    H_{\gamma,S}\coloneqq (\gamma_1,\ldots,\gamma_{r-1},e)\cdot \Delta(C_\Gamma(S))\cdot (\gamma_1^{-1},\ldots,\gamma_{r-1}^{-1},e) .
\]
Here $C_\Gamma(S)$ denotes the centraliser of $S$ in $\Gamma$.
For $b\in \Gamma$, we write $H_{\gamma,b}$ instead of $H_{\gamma,\{b\}}$. Denote the element $\underline{e}\coloneqq (e,\ldots,e)\in \Gamma^{r-1}$ and note that $H_{\underline{e},e}=\Delta(\Gamma)$.
The elementary proof of the following lemma is omitted.

\begin{lemma}\label{lem:HgOperations}
    Let $\gamma=(\gamma_1,\ldots,\gamma_{r-1}), \delta=(\delta_1,\ldots,\delta_{r-1})\in \Gamma^{r-1}$ and $S,T\subset \Gamma$ be subsets. The following hold:
    \begin{enumerate}[label={\rm{(\roman*)}}]
        \item For every $g=(g_1,\ldots,g_r)\in \Gamma^r$ we have 
        \[
            g\cdot H_{\gamma,S}\cdot g^{-1}= H_{\gamma',S'} ,
        \]
        where $\gamma'=(g_1\gamma_1g_r^{-1},\ldots,g_{r-1}\gamma_{r-1}g_r^{-1})\in \Gamma^{r-1}$ and $S'=\{g_rsg_r^{-1}\in \Gamma \ |\ s\in S\}\subset \Gamma$;
        \item $H_{\gamma,S}\cap H_{\delta,T}=H_{\gamma,S\cup T\cup \{\delta_1^{-1}\gamma_1,\ldots,\delta_{r-1}^{-1}\gamma_{r-1}\}}$;
        \item $N_{\Gamma^r}(H_{\gamma,S})=\{(\gamma_1k_1h\gamma_1^{-1},\ldots,\gamma_{r-1}k_{r-1}h\gamma_{r-1}^{-1},h)\in \Gamma^r\ |\ h\in N_\Gamma(C_\Gamma(S)), k_1,\ldots,k_{r-1}\in C_\Gamma(C_\Gamma(S))\}$.
    \end{enumerate}
\end{lemma}

Let $\calf _1\subset \cald$ be the families of subgroups of $\Gamma^r$ defined as
\begin{align*}
    \cald & \coloneqq \calf \spann{\{\Delta(\Gamma)\}} ; \\
    \calf _1 & \coloneqq \calf \spann{\{H_{\gamma,b}\ |\ \gamma\in \Gamma^{r-1},b\in \Gamma\setminus\{e\}\}} .
\end{align*}

\begin{lemma}\label{lem.HnVanish}
    Let $\Gamma$ be a group as in Theorem~\ref{thmx.main} and let $\underline{e}=(e,\dots,e)\in\Gamma^{r-1}$.  Then for $n=\cd(\Gamma^r)$ and every $\calo_\cald(\Gamma^r)$-module $A$, we have
    \[ H^n_{\Gamma^r}(\Gamma^r\times_{H_{\underline{e},e}} E_{\calf_1|_{H_{\underline{e},e}}}(H_{\underline{e},e});A)=0. \]
\end{lemma}
\begin{proof}
We have that conditions $(M_{\TR\subset \calf _1|_{H_{\underline{e},e}}})$ and $(NM_{\TR\subset \calf _1|_{H_{\underline{e},e}}})$ hold for the group $H_{\underline{e},e}$.  To see this we follow the same argument as in~\cite[Lemma 3.5]{Li2021}, using Lemma~\ref{lem:HgOperations} instead of~\cite[Lemma 3.1]{Li2021}. By Theorem~\ref{prop.LW}\ref{prop.LW.NM}, we obtain an $H_{\underline{e},e}$-pushout
    \begin{equation}\label{eqn.pushout.lemma}
    \begin{tikzcd}
\coprod_{H_{\underline{e},b\in\calm}} H_{\underline{e},e}\times_{H_{\underline{e},b}}E(H_{\underline{e},b}) \arrow[r] \arrow[d] & \arrow[d] E(H_{\underline{e},e}) \\
\coprod_{H_{\underline{e},b\in\calm}} H_{\underline{e},e}/H_{\underline{e},b} \arrow[r] & E_{\calf_1|_{H_{\underline{e},e}}}(H_{\underline{e},e}) ,
\end{tikzcd}\end{equation}
where $\calm$ is a complete set of representatives of conjugacy classes of maximal elements in $\calf_1|_{H_{\underline{e},e}}\backslash \TR$.  Since $\cd(H_{\underline{e},e})<n$ and $\cd(H_{\underline{e},b})<n-1$ for $b\in\Gamma\backslash\{e\}$, the Mayer--Vietoris sequence for $H^\ast_{H_{\underline{e},e}}(-;A)$ applied to the pushout~\eqref{eqn.pushout.lemma} yields the lemma.
\end{proof}

\begin{proof}[Proof of Theorem~\ref{thmx.main}]
Throughout the proof let $n=\cd(\Gamma^r)$.  
We show that the $\Gamma^r$-map $E(\Gamma^r)\to E_\calD(\Gamma^r)$ induces a surjective map $H^n_{\Gamma^r}(E_\calD(\Gamma^r);A)\to H^n_{\Gamma^r}(E(\Gamma^{r});A)$ for every $\mathcal{O}_\calD(\Gamma^r)$-module $A$. Then the theorem follows from Theorem~\ref{thm.Farber-Oprea} and the upper bound $\TC_r(\Gamma)\le n$ (see e.g.~\cite[(4)]{FO2019}).

First, observe that condition $(M_{\TR\subset \calf _1})$ holds and that moreover, for $\gamma\in \Gamma^{r-1}$ and $b\in \Gamma\setminus\{e\}$ there is an isomorphism $N_{\Gamma^r}(H_{\gamma,b})\cong C_\Gamma(b)^r$.  To see this we follow the same argument as in~\cite[Lemma 3.5]{Li2021}, using Lemma~\ref{lem:HgOperations} instead of~\cite[Lemma 3.1]{Li2021}.  It follows that for every $\calo_\cald(\Gamma^r)$-module $A$, we have
\[
    H^{n}_{\Gamma^r}(\Gamma^r\times_{N_{\Gamma^r}(H_{\gamma,b})}E(N_{\Gamma^r}(H_{\gamma,b}));A)=0.
\]
By Theorem~\ref{prop.LW}\ref{prop.LW.TRIV}, we obtain a $\Gamma^r$-pushout
\begin{equation}\label{eqn.pushout1}\begin{tikzcd}
\coprod_{H_{\gamma,b}\in\calm} \Gamma^r\times_{N_{\Gamma^r}(H_{\gamma,b})}E(N_{\Gamma^r}(H_{\gamma,b})) \arrow[r] \arrow[d] & \arrow[d] E(\Gamma^r)  \\
\coprod_{H_{\gamma,b}\in\calm} \Gamma^r\times_{N_{\Gamma^r}(H_{\gamma,b})}E(W_{\Gamma^r}(H_{\gamma,b})) \arrow[r] & E_{\calf_1}(\Gamma^r),
\end{tikzcd}\end{equation}
where $\mathcal{M}$ is a complete set of representatives of conjugacy classes of maximal elements in $\F_1\setminus \TR$.
Applying the Mayer--Vietoris sequence for $H^\ast_{\Gamma^r}(-;A)$ to the pushout~\eqref{eqn.pushout1} shows that the induced map $H^{n}_{\Gamma^r}(E_{\F_1}(\Gamma^r);A) \to H^{n}_{\Gamma^r}(E(\Gamma^r);A)$ is surjective.

Second, conditions $(M_{\calf _1\subset \cald})$ and $(NM_{\calf _1\subset \cald})$ hold by the same argument as in~\cite[Lemma 3.2]{Li2021}, using Lemma~\ref{lem:HgOperations} instead of~\cite[Lemma 3.1]{Li2021}. By Theorem~\ref{prop.LW}\ref{prop.LW.NM}, we obtain a $\Gamma^r$-pushout
\begin{equation}\label{eqn.pushout2}\begin{tikzcd}
\Gamma^r\times_{H_{\underline{e},e}}E_{\calf_1|_{H_{\underline{e},e}}}(H_{\underline{e},e}) \arrow[r] \arrow[d] & \arrow[d] E_{\calf_1}(\Gamma^r) \\
\Gamma^r/H_{\underline{e},e} \arrow[r] & E_\cald(\Gamma^r).
\end{tikzcd}\end{equation}
Lemma~\ref{lem.HnVanish} and the Mayer--Vietoris sequence for $H^\ast_{\Gamma^r}(-;A)$ applied to the pushout~\eqref{eqn.pushout2} yield that the map $H^{n}_{\Gamma^r}(E_\calD (\Gamma^r);A)\to H^{n}_{\Gamma^r}(E_{\F_1} (\Gamma^r);A)$ is surjective.

Together, the map $H^n_{\Gamma^r}(E_\cald (\Gamma^r); A)\to H^n_{\Gamma^r}(E(\Gamma^r); A)$ is surjective for every $\calo_\cald(\Gamma^r)$-module $A$. This finishes the proof.
\end{proof}

\begin{proof}[Proof of Corollary~\ref{corx.growth}]
    For groups $\Gamma$ of type $F$, we have $\cd(\Gamma^r)=r\cdot \cd(\Gamma)$ by \cite[Corollary 2.5]{Dranishnikov19}. The result now follows from Theorem~\ref{thmx.main}.
\end{proof}

\begin{proof}[Proof of Corollary~\ref{corx.hyp}]
    The result follows from Corollary~\ref{corx.growth} using the fact that torsion-free hyperbolic groups are of type $F$ (see e.g.~\cite[Corollary 3.26 on page 470]{Bridson-Haefliger99}). 
\end{proof}

\AtNextBibliography{\small}
\printbibliography

\end{document}